\documentclass[11pt]{amsart}
\usepackage{amsmath}
\usepackage{amssymb}
\usepackage{amsthm}
\usepackage{mathrsfs}
\usepackage{amsfonts}
\usepackage{pxfonts}
\usepackage{bm}
\usepackage{cite}
\usepackage{enumerate}
\usepackage{units}
\usepackage{color,contour}
\usepackage{varioref}
\usepackage{syntonly}
\usepackage{stdclsdv}
\usepackage{showlabels}
\usepackage{alltt}
\theoremstyle{plain}
\newtheorem{theo}{Theorem}
\newtheorem*{theo*}{Theorem}

\theoremstyle{definition}

\newtheorem{lem}{Lemma}[section]
\newtheorem{defi}{Definition}[section]
\newtheorem{conj}{Conjecture}[section]

\newtheorem{prop}{Proposition}[section]
\newtheorem{cor}{Corollary}[section]

\theoremstyle{remark}

\newtheorem*{ack}{Acknowledgements}
\newtheorem*{rem*}{Remark}

\begin{document}

\title{On commuting Tonelli Hamiltonians:  Time-periodic case}
\author{xiaojun cui}
\address{Xiaojun Cui \endgraf
Department of Mathematics,\endgraf Nanjing University,\endgraf
Nanjing, 210093,\endgraf Jiangsu Province, \endgraf People's
Republic of China}.
\curraddr{Mathematisches Institut,\endgraf
 Albert-Ludwigs-Universit\"{a}t of Freiburg,\endgraf
  Eckerstrasse 1, 79104, \endgraf
  Freiburg im Breisigau, Germany.}

 \email{xjohncui@gmail.com,xjohncui@yahoo.com}

\thanks{Supported by National Natural Science Foundation of China
 (Grant 10801071) and  research fellowship for postdoctoral researchers
 from the Alexander von Humboldt Foundation.}

\abstract We show that the Aubry sets, the Ma\~{n}\'{e} sets and the barrier functions are the same for two commuting time-periodic Tonelli
Hamiltonians.
\endabstract
\maketitle

\section*{Introduction}

In \cite{CL}, the relations  between the dynamics of two commuting time-independent Tonelli
Hamiltonians are studied, see also \cite{Za}. In this article, we will study the time-periodic case. Although the ideas here are almost the same as the ones in \cite{CL}, I think it is still interesting to write them down concretely, since lack of energy integral makes things  more or less different (for example, failure of convergence of the Lax-Oleinik semi-group \cite{FM}). Moreover, the proofs of some results here are indeed different from or more difficult than the ones in the autonomous case. Furthermore, comparing with the autonomous case, we can derive almost nothing about the dynamics of the sum Hamiltonian from the dynamics of two given commuting time-periodic Tonelli Hamiltonians.

\section{Statement of results}
Let $M$ be a closed, connected $C^{\infty}$ Riemannian manifold. Let $TM$ and $T^{*}M$ be the tangent bundle and the
cotangent bundle of $M$ respectively. In local coordinates, we may
express them as
$$TM=\big{\{}(q,\dot{q}): q \in T_qM \big{\}}$$
and
$$T^{*}M=\big{\{}(q,p): p \in T^{*}_qM\big{\}}$$
accordingly.

 Let $\mathbb{T}=\mathbb{R}/\mathbb{Z}$ be the circle. A $C^2$ function $H:T^{*}M \times \mathbb{T} \rightarrow \mathbb{R}$ is
called a (time 1-periodic) \emph{Tonelli Hamiltonian} if $H$ satisfies the following
conditions:

 $\bullet$ $H$ is  fiberwise strictly convex, i.e., the fiberwise Hessian $\frac{\partial ^2 H}{\partial
 p^2}$ is positively definite for every $(q,p,[t]) \in T^*M \times \mathbb{T}$;

$\bullet$ $H$ has superlinear growth, i.e., $\frac{H(q,p,[t])}{|p|}
\rightarrow +\infty$ as $|p| \rightarrow +\infty$;

$\bullet$ Completeness, i.e., solutions of the Hamilton equation are defined on the whole $\mathbb{R}$,
where $[t]= t \text{ mod } 1$ for $t \in \mathbb{R}$ and $ |\cdot |$
is the norm on $T^*_qM$ induced by the Riemannian metric on $M$.

Thanks to  Mather
theory \cite{Man},\cite{Mat1},\cite{Mat2} and  its weak KAM approach
\cite{Fa3}, the dynamics of the Hamilton flow
$\phi^t_H$ is well understood, in the case that $H$ is a Tonelli Hamiltonian.

 Let $\{\cdot,\cdot\}$ be the Poisson bracket. We also introduce another bracket $[\cdot, \cdot]$, defined as
$$[H_1, H_2]=\{H_1,H_2\}+\frac{\partial H_1}{\partial t}-\frac{\partial H_2}{\partial t}.$$
We say that $H_1$ and $H_2$ are \emph{commuting} (or, in involution) if $[H_1,H_2]=0$. In this article, we will consider the relations between the dynamics of Hamiltonian flow of $H_1$ and of $H_2$ under the assumption  $[H_1,H_2]=0$.

Let us recall some fundamental notations of Mather theory \cite{Mat1},\cite{Mat2}. For a Tonelli Hamiltonian $H$, let $L_{H}$ be the Lagrangian
associated to $H$ by Legendre transformation:
$$L_H(q,\dot q,[t])=p \dot q-H(q,p,[t]), $$ here $p$ and $\dot q$ are related by $\dot q
=\frac{\partial H(q,p,[t])}{\partial p}.$ Throughout this paper, we use
$\mathcal{L}_H$ to denote the (time-dependent) Legendre transformation from the tangent
bundle $TM$ to the cotangent bundle $T^*M$, i.e.,
$$\mathcal{L}_H(\dot q)=p \iff \dot q =\frac{\partial H(q,p,[t])}{\partial p},$$
where the time-dependence will be apparent in the context.
Let
$$\alpha_H(0)=-\min_{\mu}\int L_{H} d\mu,$$
where the
minimum is taken over all invariant (under the Euler-Lagrange flow
$\phi^t_L$ of $L$, which is defined on the extended phase space $TM \times \mathbb{T}$) Borel probability measures. We say that an
invariant Borel probability measure $\mu$ is minimal if $\int
L_H d \mu=-\alpha_H(0)$. Let
$\mathfrak{M}$ be the set of minimal measures. Let the Mather
set be
$$\dot{M}_{H}=\dot{M}_{L_H}= \text{closure of }
\big{\{}\cup_{\mu \in \mathfrak{M}} \text{ support of }
\mu\big{\}}\subset TM \times \mathbb{T}.$$
  The set
$$\stackrel{\scriptscriptstyle \ast}{M}_{H}=
\stackrel{\scriptscriptstyle \ast}{M}_{L_H}=
\mathcal{L}_H\dot{M}_{H} \subset T^*M \times \mathbb{T}$$ is also called the Mather set. Throughout this paper, let $\pi$ be the projection of $T^{*}M \times \mathbb{T}$, $T^{*}(M \times \mathbb{T})$ or $TM \times \mathbb{T}$ along the associated fibers onto $M \times \mathbb{T}$ or $M$, according to the circumstance.  The projection of $\stackrel{\scriptscriptstyle \ast}{M}_{H}$ (or
$\dot{M}_{H}$, equivalently) into $M \times \mathbb{T}$ is called the \emph{projected} Mather
set. We denote it by $M_{H}=M_{L_H}.$

For any $T >0$ and any two points $(q_1,[t_1]), (q_2,[t_2]) \in M \times \mathbb{T}$,   let
$$h^T_{H}((q_1,[t_1]),(q_2,[t_2]))=
\inf_{\gamma} \int^{t_2}_{t_1}
\big{(}L_H+\alpha_H(0)\big{)}(\gamma(t),\dot{\gamma}(t),[t])dt,$$
where minimum is taken over all absolutely continuous curves $\gamma:[t_1,t_2]
\rightarrow M$  with $\gamma(t_1)=q_1,\gamma(t_2)=q_2$ and $t_2-t_1 \geq T$. Let$$
h_{H}((q_1,[t_1]),(q_2,[t_2]))=\lim_{T \rightarrow +\infty}h^T_{H}((q_1,[t_1]),(q_2,[t_2])).$$
Let
\begin{eqnarray*}&&\rho_H((q_1,[t_1]),(q_2,[t_2]))\\
&=&\rho_{L_H}((q_1,[t_1]),(q_2,[t_2]))\\
&=& h_{H}((q_1,[t_1]),(q_2,[t_2]))+h_{H}((q_2,[t_2]),(q_1,[t_1])).
\end{eqnarray*}

We say that  an absolutely continuous curve
$\gamma:\mathbb{R} \rightarrow M$ is a \emph{weak minimizer}, if for any compact
interval $[a,b]$ and any absolutely continuous curve $\gamma_1:[a,b]
\rightarrow M$ such that $\gamma_1(a)=\gamma(a)$,
$\gamma_1(b)=\gamma(b)$,  we have
$$\int^b_aL_H(\gamma(t),\dot{\gamma}(t),[t])dt \leq
\int^{b}_{a}L_H(\gamma_1(t),\dot{\gamma_1}(t),[t])dt.$$
We define the set of weak minimizers to be
$$\dot{W}_{H}=\dot{W}_{L_H}=\cup \big{\{}(\gamma(t),
\dot{\gamma}(t), [t]):\gamma \text{ is a weak minimizer}\big{\}}.$$ Thus
$\dot{W}_H \subset TM \times \mathbb{T}$.
Let
$$\stackrel{\scriptscriptstyle \ast}{W}_{H}=
\stackrel{\scriptscriptstyle \ast}{W}_{L_H}=
\mathcal{L}_H\dot{W}_{H}.$$ Then, $\stackrel{\scriptscriptstyle
\ast}{W}_H \subset T^*M \times \mathbb{T}$.

 We say that an absolutely continuous curve
$\gamma:\mathbb{R} \rightarrow M$ is  a \emph{minimizer}, if for any compact
interval $[a,b]$ and any absolutely continuous curve $\gamma_1:[a_1,b_1]
\rightarrow M$ such that $\gamma_1(a_1)=\gamma(a)$,
$\gamma_1(b_1)=\gamma(b)$, $a_1-a \in \mathbb{Z},b_1-b \in \mathbb{Z},$ we have
$$\int^b_a(L_H+\alpha_H(0))(\gamma(t),\dot{\gamma}(t),[t])dt \leq
\int^{b_1}_{a_1}(L_H+\alpha_H(0))(\gamma_1(t),\dot{\gamma_1}(t),[t])dt.$$
We define the Ma\~{n}\'{e} set
$$\dot{N}_{H}=\dot{N}_{L_H}=\cup \big{\{}(\gamma(t),
\dot{\gamma}(t), [t]):\gamma \text{ is a minimizer}\big{\}}\subset TM \times \mathbb{T}.$$
Let
$$\stackrel{\scriptscriptstyle \ast}{N}_{H}=
\stackrel{\scriptscriptstyle \ast}{N}_{L_H}=
\mathcal{L}_H\dot{N}_{H}\subset T^*M \times \mathbb{T},$$
and it is also called the Ma\~{n}\'{e} set.

Let $\gamma: \mathbb{R} \rightarrow M$ be a minimizer. Let $(q_{\alpha},[t_{\alpha}])$
be in the $\alpha$-limit set of $(\gamma(t),[t])$ and $(q_{\omega},[t_{\omega}])$ be in the $\omega$-limit set
 of $(\gamma(t),[t])$. If $\rho_{H}((q_{\alpha},[t_{\alpha}]),(q_{\omega},[t_{\omega}]))=0$, we say that $\gamma$
is a \emph{regular minimizer}. Let
$$\dot{A}_{H}=\dot{A}_{L_H}=\cup \big{\{}(\gamma(t),
\dot{\gamma}(t),[t]):\gamma \text{ is a regular minimizer}\big{\}}.$$
Clearly, $\dot{A}_H \subseteq \dot{N}_H$.
We define
$$\stackrel{\scriptscriptstyle \ast}{A}_{H}=
\stackrel{\scriptscriptstyle \ast}{A}_{L_H}=
\mathcal{L}_H\dot{A}_{H}.$$ Let the \emph{projected} Aubry set $$A_{H}=A_{L_H}$$ and the \emph{projected} Ma\~{n}\'{e} set
$$N_{H}=N_{L_H}$$ be the projections of
$$\stackrel{\scriptscriptstyle \ast}{A}_{H}=
\stackrel{\scriptscriptstyle \ast}{A}_{L_H} \text{ and
}\stackrel{\scriptscriptstyle \ast}{N}_{H}=
\stackrel{\scriptscriptstyle \ast}{N}_{L_H}$$ into $M \times \mathbb{T}$ respectively.

It is known that $\rho_H$ is a pseudo metric on $A_{H}$, and the induced metric space is denoted by $(\bar{A}_H, \rho_H)$. This is the so-called \emph{quotient} Aubry set.

We have the following inclusions:
$$\stackrel{\scriptscriptstyle \ast}{M}_{H} \subseteq \stackrel{\scriptscriptstyle \ast}
{A}_{H}\subseteq \stackrel{\scriptscriptstyle \ast}{N}_{H}\subseteq \stackrel{\scriptscriptstyle \ast}{W}_{H}.$$

In this article, it is also convenient to define  some associated  sets in the symplectic manifold $(T^*(M \times \mathbb{T}), dp \wedge dq +d\kappa \wedge dt)$. We define\\
the \emph{extended} Mather set to be
$$\stackrel{\scriptscriptstyle \ast}{M}^e_{H}=\{(q,p,[t],\kappa):(q,p,[t]) \in \stackrel{\scriptscriptstyle \ast}{M}_{H} \text{ and } H(q,p,[t])+\kappa=\alpha_H(0)\};$$
the \emph{extended} Aubry set to be
$$\stackrel{\scriptscriptstyle \ast}{A}^e_{H}=\{(q,p,[t],\kappa):(q,p,[t]) \in \stackrel{\scriptscriptstyle \ast}{A}_{H} \text{ and } H(q,p,[t])+\kappa=\alpha_H(0)\};$$
the \emph{extended} Ma\~{n}\'{e} set to be
$$\stackrel{\scriptscriptstyle \ast}{N}^e_{H}=\{(q,p,[t],\kappa):(q,p,[t]) \in \stackrel{\scriptscriptstyle \ast}{N}_{H} \text{ and } H(q,p,[t])+\kappa=\alpha_H(0)\}.$$

It should be mentioned that these sets have been studied previously, for example, in  \cite{Be1}.
Clearly, $\pi:\stackrel{\scriptscriptstyle \ast}{A}^e_{H} \rightarrow A_H$ is also a bi-Lipschitz map.

Now we define \emph{Ma\~{n}\'{e} potential}:
$$\phi_{H}((q_1,[t_1]),(q_2,[t_2]))=\inf \int^{t^{\prime}_2}_{t^{\prime}_1} (L_H+\alpha_H(0))(\gamma(t),\dot{\gamma}(t),[t])dt,$$
where infimum is taken over all absolutely continuous curves $\gamma:[t^{\prime}_1,t^{\prime}_2] \rightarrow M$ with $\gamma(t^{\prime}_1)=q_1,\gamma(t^{\prime}_2)=q_2$, $[t^{\prime}_1]=[t_1]$ and $[t^{\prime}_2]=[t_2]$. Clearly, $\phi_{H} \leq h_{H}$.

Now we say \cite{CIS}  that $\phi:M \times \mathbb{T} \rightarrow \mathbb{R}$ is a \emph{backward} (resp. \emph{forward}) weak KAM solution if

1). $\phi$ is $L_H+\alpha(0)$ \emph{dominated}, i.e.,
$$\phi(q_2,[t_2])-\phi(q_1,[t_1]) \leq \phi_{H}((q_1,[t_1]),(q_2,[t_2])).$$
We use the notation $\phi \prec L_H+\alpha_H(0)$ to denote this relation.

2). For every $(q,\tau) \in M \times \mathbb{R}$, there exists a curve $\gamma:(-\infty,\tau] \rightarrow M$ (resp. $\gamma:[\tau, \infty) \rightarrow M)$ with $\gamma(\tau)=q$ such that $\phi(q,[\tau])-\phi(\gamma(s), [s])=\int^{\tau}_{s}(L_H+\alpha_H(0))(\gamma(t),\dot{\gamma}(t),[t])dt$ (resp. $\phi(\gamma(s), [s])-\phi(q,[\tau])=\int^{s}_{\tau}(L_H+\alpha_H(0))(\gamma(t),\dot{\gamma}(t),[t])dt$). In this case, $\gamma$ is called to be a \emph{backward minimizer} (resp. \emph{forward minimizer}) from $(q,[\tau])$.

Let $\mathcal{S}^-_H$ (resp. $\mathcal{S}^+_H$) be the set of backward (resp. forward) weak KAM solutions.

Now for $t_1 \leq t_2 \in \mathbb{R}$, we introduce two families of nonlinear operators
$(T^-_{H,t_1,t_2})$ and $(T^+_{H,t_1,t_2})$
respectively, i.e., the so-called Lax-Oleinik
operators. To define them, let us fix a function $u \in C^0(M,\mathbb{R})$.  For $q \in M$, we set
$$T^-_{H,t_1,t_2}u(q)=\inf_{\gamma}\Big{\{}u(\gamma(t_1))+\int^{t_2}_{t_1} (L_H+\alpha_H(0))
(\gamma(t),\dot{\gamma}(t),[t])dt\Big{\}},$$ where the infimum is taken
over all absolutely continuous curves $\gamma:[t_1,t_2]\rightarrow M$
such that $\gamma(t_2)=q$. Also, for $q \in M$, we set
$$T^+_{H,t_1,t_2}u(q)=\sup_{\gamma}\Big{\{}u(\gamma(t_2))-\int^{t_2}_{t_1} (L_H+\alpha_H(0))
(\gamma(t),\dot{\gamma}(t),[t])dt\Big{\}},$$ where the supremum is taken
over all absolutely continuous curves $\gamma:[t_1,t_2]\rightarrow M$
such that $\gamma(t_1)=q$.

So far, we can state our main results as follows:

\begin{theo}
Let $H_1,H_2$ be  two Tonelli Hamiltonians $H_1,H_2$. If
$[H_1,H_2]=0$, then
$$T^-_{H_1,t_1,t_3}T^-_{H_2,t_0,t_1}u=T^-_{H_2,
t_2,t_3}T^-_{H_1,t_0,t_2}u, \,\,\,\,\,\,\,\,
T^+_{H_1,t_1,t_3}T^+_{H_2,t_0,t_1}u=T^+_{H_2,
t_2,t_3}T^+_{H_1,t_0,t_2}u$$ for any $u \in
C^0(M,\mathbb{R})$ and
$t_0 \leq t_1\leq t_3$ and $t_0 \leq t_2\leq t_3$ with $t_3-t_1=t_2-t_0$ (and, certainly, $t_1-t_0=t_3-t_2$).
\end{theo}

\begin{rem*}
The autonomous version of Theorem 1 appeared in \cite{BT}, see also \cite{CL}, \cite{Za} for a variational discussion.

\end{rem*}

\begin{theo}
Let $H_1, H_2$ be two Tonelli Hamiltonians. If $[H_1,H_2]=0$, then
$\mathcal{S}^{+}_{H_1}=\mathcal{S}^{+}_{H_2}$ and
$\mathcal{S}^{-}_{H_1}=\mathcal{S}^{-}_{H_2}$.
\end{theo}

Now we recall the definition of barrier functions \cite{Mat2}. The first
barrier function
$$B_{H}(q,[t])=h_{H}((q,[t]),(q,[t]));$$
 the second barrier function
$$b_{H}(q,[t])=\min_{\xi ,\zeta \in
A_{H}}\{h_{H}(\xi, (q,[t]))+h_{H}((q,[t]),\zeta)
-h_{H}(\xi,\zeta)\}.$$
.
\begin{theo}
Let $H_1, H_2$ be two Tonelli Hamiltonians. If $[H_1,H_2]=0$, then
 $B_{H_1}(q,[t])=B_{H_2}(q,[t])$ and
$b_{H_1}(q,[t])=b_{H_2}(q,[t])$.
\end{theo}

\begin{theo}
Let $H_1, H_2$ be two Tonelli Hamiltonians. If $[H_1,H_2]=0$, then
$\stackrel{\scriptscriptstyle
\ast}{A}^e_{H_1}=\stackrel{\scriptscriptstyle \ast}{A}^e_{H_2}$
and $\stackrel{\scriptscriptstyle
\ast}{N}^e_{H_1}=\stackrel{\scriptscriptstyle \ast}{N}^e_{H_2}$.
\end{theo}

This article is organized as follows. In section 2, we recall some fundamental properties about weak KAM solutions and viscosity solutions (subsolutions, supersolutions, and so on), which will be useful in the sequel. Then, we will prove one Theorem in each section sequently. In section 7, we pose a conjecture, which is motivated by a previous autonomous result (Theorem 6, \cite{CL}). In section 8, we provide a remark, which illustrates some differences  between the time-periodic case and the autonomous case.

\section{Weak KAM solutions and time 1-periodic viscosity solutions}

In this section, we will recall some relations between  backward weak solutions and  time 1-periodic viscosity solutions, for the  Hamilton-Jacobi equation associated to a  time 1-periodic Tonelli Hamiltonian. We will also recall some properties of Lax-Oleinik operators.  The autonomous version of these results is well known, thanks to the weak KAM theory \cite{Fa3}. For the time-periodic case, nice expositions appeared in \cite{BR},\cite{Be4}, and we will summarize them in this section with slight modifications. To the best of the author's knowledge, it is Zhukovskaya who first realized that viscosity solutions coincide with minimax solutions (i.e., backward weak KAM solutions in Tonelli systems) for convex systems (including Tonelli ones) \cite{Zh1}, \cite{Zh2}.

Let $H$ be a time 1-periodic Tonelli Hamiltonian. Consider the Hamilton-Jacobi equation
   $$H(q,d_q\phi,[t])+d_t\phi=\alpha_H(0).$$
Then we  have, (see for example ,\cite{Be4},\cite{BR},\cite{CS},\cite{Zh1}, \cite{Zh2})
\begin{prop}
A function $\phi:M \times \mathbb{T} \rightarrow \mathbb{R}$ is a backward weak KAM solution if and only if $\phi$ is a Lipschitz viscosity solution. Moreover, for any $u \in C^0(M, \mathbb{R})$, we have
$$\phi(q,t)=\liminf_{n \rightarrow \infty}T^-_{H,s,t+n}u(q)$$
is a time 1-periodic Lipschitz viscosity solution. Also, a Lipschitz function
$\phi: M \times \mathbb{T} \rightarrow \mathbb{R}$ is viscosity solution if and only if for any $t_1 <t_2$, we have $T^-_{H, t_1,t_2}\phi(\cdot, [t_1])=\phi(\cdot, [t_2])$.
\end{prop}

Let us recall some  properties of Lax-Oleinik operator \cite{BR},\cite{Be4}:

$\bullet$ Markov property, this means that
$$T^-_{H,t_1,t_2}T^-_{H,t_2,t_3}u=T^-_{H,t_1,t_3}u$$
for any $u \in C^0(M, \mathbb{R})$ and $t_3 \geq t_2 \geq t_1\in \mathbb{R}$.

$\bullet$ Contraction, i.e.,

$$\|T^{-}_{H,t_1,t_2}u_1-T^{-}_{H,t_1,t_2}u_2\|_{\infty} \leq \|u_1-u_2\|_{\infty},$$
for any two continuous functions $u_1,u_2 \in C^0(M,\mathbb{R})$ and $t_2
\geq t_1$.

$\bullet$ $T^{-}_{H,t_1,t_2}$ is  compact, order-preserving and $T^{-}_{H,t_1,t_2}(k+u)=k+T^{-}_{H,t_1,t_2}u$ for any real number $k$ and any $u \in C^0(M,\mathbb{R})$.

Certainly, there are also analogous properties for $T^+_{H,t_1,t_2}$.

\section{Proof of Theorem 1}

Let $H(q,p,[t])$ be a time 1-periodic Tonelli Hamiltonian. We regard that the Hamiltonian flow $\phi^t_H$ is defined on the extended phase space $T^*M \times \mathbb{T}$. We may also associate it to an autonomous Hamiltonian
$$\tilde H(q,p,[t],\kappa)=\kappa+H(q,p,[t]),$$
defined on symplectic manifold $(T^*(M \times \mathbb{T}), dp \wedge dq+d \kappa \wedge dt)$. The associated Hamiltonian flow is
$$\phi^s_{\tilde H}:(q,p,[t],\kappa) \rightarrow (\phi^s_H(q,p,[t]),\kappa+H(q,p,[t])-H(\phi^s_H(q,p,[t]))).$$

Clearly,  $[H_1,H_2]=0$ if and only if $\{\tilde{H_1},\tilde{H_2}\}=0.$

For simplicity of notations, we also introduce the notation $h^{[t_1,t_2]}_H$, which is defined  by
$$h^{[t_1,t_2]}_H(q_1,q_2)=\inf_{\gamma}\int^{t_2}_{t_1}(L_{H}
+\alpha_{H}(0))(\gamma(t),\dot{\gamma}(t),[t])dt,$$
where $ t_2 > t_1 \in \mathbb{R}$ and  $\gamma$ ranges over all absolutely continuous curves with $\gamma(t_1)=q_1,\gamma(t_2)=q_2$.

We only  prove the first equality, and the second equality in the theorem can be proved
analogously.

For any point $q_0 \in M $ and $t_0,t_1,t_2,t_3$ as the conditions of Theorem 1 required, we will prove that
$$T^-_{H_1,t_1,t_3}T^-_{H_2, t_0,t_1}u(q_0)=T^-_{H_2,t_2,t_3}T^-_{H_1,t_0,t_2}u(q_0).$$
By the definition,
\begin{eqnarray*}&&T^-_{H_1,t_1,t_3}T^-_{H_2,t_0,t_1}u(q_0)
\\&=&\min_{x
\in M}(T^-_{H_2,t_0,t_1}u(x)+h^{[t_1,t_3]}_{H_1}(x,q_0))\\
&=& \min_{x, y \in
M}(u(y)+h^{[t_0,t_1]}_{H_2}(y,x)
+h^{[t_1,t_3]}_{H_1}(x, q_0)).
\end{eqnarray*}
Clearly, there exist two points $x_0,y_0$ such that
$$T^-_{H_1,t_1,t_3}T^-_{H_2, t_0,t_1}u(q_0)=u(y_0)+h^{[t_0,t_1]}_{H_2}(y_0, x_0)+h^{[t_1,t_3]}_{H_1}(x_0, q_0).$$
We assume that $$\gamma_1:[t_0,t_1] \rightarrow M \text{ and }
\gamma_2:[t_1,t_3] \rightarrow M$$ are two weak minimizers that reach
$h^{[t_0,t_1]}_{H_2}(y_0,x_0),h^{[t_1,t_3]}_{H_1}(x_0,q_0)$ respectively. In other words, these two curves $\gamma_1$, $\gamma_2$ satisfy: $\gamma_1(t_0)=y_0, \gamma_1(t_1)=\gamma_2(t_1)=x_0, \gamma_2(t_3)=q_0$ and
$$\int^{t_1}_{t_0}(L_{H_2}+\alpha_{H_2}(0))(\gamma_1(t), \dot{\gamma_1}(t),[t])dt=h^{[t_0,t_1]}_{H_2}(y_0,x_0),$$
$$\int^{t_3}_{t_1}(L_{H_1}+\alpha_{H_1}(0))(\gamma_2(t), \dot{\gamma_2}(t),[t])dt=h^{[t_1,t_3]}_{H_1}(x_0,q_0).$$
Now we have
\begin{lem}
$\mathcal{L}_{H_2}(\dot{\gamma}_1(t_1))=\mathcal{L}_{H_1}(\dot{\gamma}_2(t_1))$.
\end{lem}
The proof of this lemma is just a standard variational discussion. In autonomous case, there is a proof in \cite{CL}. For time-periodic case, only some trivial modifications are needed. So, the proof is omitted.

We denote
$$\mathcal{L}_{H_2}(\dot{\gamma}_1(t_1))=\mathcal{L}_{H_1}(\dot{\gamma}_2(t_1))$$
by $p^{\diamond}$. Hence, if we assume that
$\mathcal{L}_{H_1}(\dot{\gamma}_2(t_3))=p_0$, then $$(y_0,
\mathcal{L}_{H_2}(\dot{\gamma}_1(t_0)),[t_0])= \phi^{t_0-t_1}_{H_2}
\phi^{t_1-t_3}_{H_1}(q_0,p_0,[t_3]).$$
Now we introduce an auxiliary variable $\kappa_0$ and regard $(q_0,p_0,[t_3],\kappa_0)$ as a point in $T^*(M \times \mathbb{T})$.
 Thus,
\begin{eqnarray*}
&&T^-_{H_1,t_1,t_3}T^-_{H_2,t_0,t_1}u(q_0)\\
&=&u(y_0)+\int(pdq-H_1dt)\big{(}\phi^t_{H_1}(x_0,p^{\diamond},[t_1])|_{[0,t_3-t_1]}\big{)}\\
&+&\int(pdq-H_2dt)\big{(}\phi^t_{H_2}(x_0,p^{\diamond},[t_1])|_{[t_0-t_1,0]}\big{)}\\
&+&(t_3-t_1)\alpha_{H_1}(0)+(t_1-t_0)\alpha_{H_2}(0)\\
&=&u(y_0)+\int(pdq+\kappa dt)\big{(}\phi^t_{\tilde {H_1}}(x_0,p^{\diamond},[t_1],\kappa^{\diamond})|_{[0,t_3-t_1]}\big{)}\\
&+&\int(pdq+\kappa dt)\big{(}\phi^t_{\tilde{ H_2}}(x_0,p^{\diamond},[t_1],\kappa^{\diamond})|_{[t_0-t_1,0]}\big{)}\\
&+&(t_3-t_1)\alpha_{H_1}(0)+(t_1-t_0)\alpha_{H_2}(0)\\
&-&(t_3-t_1)\tilde{H_1}(x_0,p^{\diamond},[t_1],\kappa^{\diamond})-
(t_1-t_0)\tilde{H_2}(x_0, p^{\diamond}, [t_1], \kappa^{\diamond}),
\end{eqnarray*}
here, and in the following, $pdq-Hdt$ is regarded as a smooth 1-form on $T^*M\times \mathbb{T}$ and  $\kappa_0$ and $\kappa^{\diamond}$ are related by $$(q_0,p_0,[t_3],\kappa_0)=\phi^{t_3-t_1}_{\tilde{H_1}}(x_0, p^{\diamond},[t_1],\kappa^{\diamond}).$$

Assume that $$\phi^{t_1-t_0}_{\tilde{H_2}}(y_0, \mathcal{L}_{H_2}(\dot{\gamma_1}(t_0)),[t_0],\kappa^*)=
(x_0,p^{\diamond},[t_1],\kappa^{\diamond}).$$

 Let
$\gamma_4:[t_2, t_3] \rightarrow M$ be the curve such that $$\pi:
\phi^t_{H_2}(q_0,p_0,[t_3])|_{[t_2-t_3,0]}=\gamma_4.$$
Similarly, let $\gamma_3:[t_0, t_2] \rightarrow M$ be the curve such that
$$\pi:
\phi^t_{H_1}(\gamma_4(t_2),\mathcal{L}_{H_2}
(\dot{\gamma_4}(t_2)),[t_2])|_{[t_0-t_2,0]}=\gamma_3.$$
Since $[H_1,H_2]=0$, then $\{\tilde{H_1},\tilde{H_2}\}=0$, and so we have
$$\phi^{t_0-t_2}_{\tilde{H_1}}\phi^{t_2-t_3}_{\tilde{H_2}} (q_0,p_0,[t_3],\kappa_0)=(y_0,\mathcal{L}_{H_2}(\dot{\gamma_1}(t_0)),[t_0],\kappa^{*}).$$

Hence,
\begin{eqnarray*}
&&T^-_{H_2,t_2,t_3}T^-_{H_1,t_0,t_2}u(q_0)\\
 &\leq& u(\pi \circ  \phi^{t_0-t_2}_{H_1}\phi^{t_2-t_3}_{H_2} (q_0,p_0,[t_3]))\\
&+& \int^{t_2}_{t_0}
(L_{H_1}+\alpha_{H_1}(0))(\gamma_3(t),\dot{\gamma_3}(t),[t])dt
\\&+& \int^{t_3}_{t_2}
(L_{H_2}+\alpha_{H_2}(0))(\gamma_4(t),\dot{\gamma_4}(t),[t])dt\\
&=&u(y_0)+\int(pdq+\kappa dt)\big{(}\phi^t_{\tilde{H_1}}(y_0,\mathcal{L}_{ H_2}(\dot{\gamma_1}(t_0)),[t_0],\kappa^{*})|_{[0,t_2-t_0]}\big{)}\\
&+&\int(pdq+\kappa dt)(\phi^t_{\tilde{H_2}}(q_0,p_0,[t_3],\kappa_0)|_{[t_2-t_3,0]})\\
&-&(t_2-t_0)\tilde{H_1}(y_0,\mathcal{L}_{H_2}(\dot{\gamma_1}(t_0)),[t_0],\kappa^{*})
-(t_3-t_2)\tilde{H_2}(q_0,p_0,[t_3],\kappa_0)\\
&+&(t_2-t_0)\alpha_{H_1}(0)+(t_3-t_2)\alpha_{H_2}(0)\\
&=&u(y_0)+\int(pdq+\kappa dt)\big{(}\phi^t_{\tilde {H_1}}(x_0,p^{\diamond},[t_1],\kappa^{\diamond})|_{[0,t_3-t_1]}\big{)}\\
&+&\int(pdq+\kappa dt)\big{(}\phi^t_{\tilde{H_2}}(y_0,\mathcal{L}_{H_2}
(\dot{\gamma_1}(t_0)),[t_0],\kappa^{*})|_{[0,t_1-t_0]} \big{)}\\
&+&(t_1-t_0)\alpha_{H_2}(0)+(t_3-t_1)\alpha_{H_1}(0)\\
&-&(t_3-t_1)\tilde{H_1}(x_0,p^{\diamond},[t_1],\kappa^{\diamond})-
(t_1-t_0)\tilde{H_2}(x_0, p^{\diamond}, [t_1], \kappa^{\diamond})\\
&=& T^-_{H_1,t_1,t_3}T^-_{H_2,t_0,t_1}u(q_0),
\end{eqnarray*}
where the first inequality follows from the definition of
Lax-Oleinik operators; the first equality follows from  direct
calculation; the second equality follows from Stokes' formula and the facts that $\{\tilde{H_1}, \tilde{H_2}\}=0$ and that both $\tilde{H_1}$ and $\tilde{H_2}$ are constants on all of these trajectories. More precisely, Stokes' formula is applied as follows:
\begin{eqnarray*}
&& \int (pdq+\kappa dt)(\phi^t_{\tilde{H_1}}(y_0, \mathcal{L}_{H_2}(\dot{\gamma_1}(t_0)),[t_0],\kappa^*)|_{[0,t_2-t_0]})\\
&+&\int (pdq+\kappa dt)(\phi^t_{\tilde{H_2}}(q_0, p_0,[t_3],\kappa_0)|_{[t_2-t_3,0]})\\
&-&\int (pdq+\kappa dt)(\phi^t_{\tilde{H_1}}(x_0, p^{\diamond},[t_1],\kappa^{\diamond})|_{[0,t_3-t_1]})\\
&-&\int (pdq+\kappa dt)(\phi^t_{\tilde{H_2}}(y_0, \mathcal{L}_{H_2}(\dot{\gamma_1}(t_0)),[t_0],\kappa^*)|_{[0,t_1-t_0]})\\
&=& \int_{\phi^t_{\tilde{H_1}}|_{[t_0-t_2,0]}
(\phi^t_{\tilde{H_2}}(q_0,p_0,[t_3],\kappa_0)|_{[t_2-t_3,0]})} dp\wedge dq+d \kappa \wedge dt\\
&=&0,
\end{eqnarray*}
here the first equality is a direct application of Stokes' formula, the second equality holds by the following reason. The tangent space to the closed region
$$\phi^t_{\tilde{H_1}}|_{[t_0-t_2]}
(\phi^t_{\tilde{H_2}}(q_0,p_0,[t_3],\kappa_0)|_{[t_2-t_3,0]})$$
(which is a $C^1$ manifold with piecewise-$C^1$ boundary or a piecewise $C^1$ curve, depending on whether $X_{\tilde{H_1}}$ and $X_{\tilde{H_2}}$ are independent somewhere, by smooth dependence of ODE on initial conditions, see for example \cite{Har}) is spanned by $X_{\tilde{H_1}}$ and $X_{\tilde{H_2}}$ (here, $X_{\tilde{H_i}}$ denotes the Hamiltonian vector field of $\tilde{H_i}$ with respect to symplectic structure $dp \wedge dq+d \kappa \wedge dt,i=1,2$). Since $\{\tilde{H_1},\tilde{H_2}\}=0$,
$$(dp \wedge dq+d \kappa \wedge dt)(X_{\tilde{H_1}},X_{\tilde{H_2}}) \equiv 0.$$

The opposite inequality can be proved similarly. So far, the proof of Theorem 1 is completed. \qed

Let $T^{\mp}_{H,1}=T^{\mp}_{H,0,1}$ and
$$T^{\mp}_{H,n}=\underbrace{T^{\mp}_{H,1} \circ \cdots \circ T^{\mp}_{H,1}}_n.$$
It is known that $\{T^{-}_{H,n}\}$ (resp. $\{T^{+}_{H,n}\}$)   composes a discrete semi-group, by the Markov property of Lax-Oleinik operators. Then, a particular form of Theorem 1 is :

\begin{prop}
Let $u \in C^{0}(M,  \mathbb{R})$, then
$$T^-_{H_1,n}T^-_{H_2,k}u=T^-_{H_2,
k}T^-_{H_1,n}u, \,\,\,\,\,\,\,\,
T^+_{H_1,n}T^+_{H_2,k}u=T^+_{H_2,
k}T^+_{H_1,n}u,$$
for any $n, k \in \mathbb{Z}^+$, here $\mathbb{Z}^+$ denotes the set of nonnegative integers.
\end{prop}

By the result of Bernard \cite{Be4}, we know that for any continuous function $u$, if $$u^*=\liminf_{n \rightarrow \infty}T^-_{H,n}u,$$
then $u^*$ is a fixed point of $T^-_{H,1}$. Denote the set of fixed points of $T^{+}_{H,1}$ (resp.$T^{-}_{H,1}$)  by $\mathcal{S}^{+}_{H,0}$ (resp. $\mathcal{S}^{-}_{H,0}$). It is known \cite{Be4} that, both two sets $\mathcal{S}^{+}_{H,0}$ and $\mathcal{S}^{-}_{H,0}$ are nonempty. Moreover, we also know that the time 1-periodic function
$$\phi^*(q,t)=T^-_{H,0,t}u^*(q), \,\,\,\,\,\, t \in \mathbb{R}$$
is a weak KAM solution (or, equivalently, time 1- periodic viscosity solution) of
$$H(q, d_q\phi,[t])+d_t \phi=\alpha_H(0),$$
for any $u^* \in \mathcal{S}^{-}_{H,0}$.
Now we claim that

\begin{prop}
$\mathcal{S}^{-}_{H_1,0}\cap \mathcal{S}^{-}_{H_2,0} \neq \emptyset, \mathcal{S}^{+}_{H_1,0}\cap \mathcal{S}^{+}_{H_2,0} \neq \emptyset.$
\end{prop}

\begin{proof}
We only prove the first relation.

Choose $u \in \mathcal{S}^-_{H_1,0}$. By Proposition 3.1, we know that
$$T^-_{H_1,n}T^-_{H_2,k}u=T^-_{H_2,
k}T^-_{H_1,n}u=T^-_{H_2,k}u$$
for any $n, k \in \mathbb{Z}^+$. So, $T^-_{H_2,k}u \in \mathcal{S}^-_{H_1,0}$ for any $k \in \mathbb{Z}^+$. Now let
$$u^*=\liminf_{k \rightarrow \infty}T^-_{H_2,k}u,$$
Then, clearly we have

$$u^* \in \mathcal{S}^-_{H_2,0}.$$

In fact, we also have
\begin{lem}
$u^* \in \mathcal{S}^-_{H_1,0}.$
 \end{lem}

\begin{proof}
(of the Lemma.)  By the definition of $\mathcal{S}^{-}_{H_1,0}$, we only need to show that for any point $q \in M$,  $$u^*(q)=T^{-}_{H_1,1}u^*(q)=\min_{q_1\in M}(u^*(q_1)+h^{[0,1]}_{H_1}(q_1,q)).$$

For simplicity of notations, we let  $u_k=T^-_{H_2,k}u$. Note that $u_k\in \mathcal{S}^-_{H_1,0}$ for each $k \in \mathbb{Z}^+$. Also, recall that the family of $u_k$ is equi-continuous \cite{Be4}.

For every $k \in \mathbb{Z}^+$ and any two  points $q,q_1 \in M$, we have $$u_k(q) \leq u_k(q_1)+h^{[0,1]}_{H_1}(q_1,q),$$
since $u_k \in \mathcal{S}^{-}_{H_1,0}$.
Then, $u^{*}(q) \leq u^{*}(q_1)+h^{[0,1]}_{H_1}(q_1,q)$. Thus, we obtain
$$u^{*}(q) \leq \inf_{q_1 \in M}(u^{*}(q_1)+h^{[0,1]}_{H_1}(q_1,q))=T^-_{H_1,1}u^{*}(q).$$

On the other hand, let us fix arbitrarily an $\epsilon >0$ and choose a subsequence  $u_{k_j} (j \geq 3)$ such that $u^{*}(q) \geq u_{k_j}(q)-\epsilon$. For each $3 \leq j \in \mathbb{Z}^+$, choose a point $q_j \in M$ such that
$$u_{k_j}(q)=u_{k_j}(q_j)+h^{[0,1]}_{H_1}(q_j,q).$$
Since $M$ is compact, we may assume that $q_j \rightarrow q^*$ by taking a subsequence if necessary.  Then, there exists a $j_0$ such that $u_{k_j}(q^*) \geq u^{*}(q^*)-\epsilon$, when $j \geq j_0$. Since the family $u_k$ is equi-continuous, there exists $j_1 \geq j_0$ such that for $j \geq j_1$,
$$u_{k_j}(q_j) \geq u^{*}(q_j)-2\epsilon.$$

Hence, for $j \geq j_1$,  we have
$$u^{*}(q) \geq u_{k_j}(q)-\epsilon = u_{k_j}(q_j)+h^{[0,1]}_{H_1}(q_j,q)-\epsilon \geq u^{*}(q_j)+h^{[0,1]}_{H_1}(q_j,q)-3\epsilon.$$
As a consequence, we have $u^{*}(q) \geq T^-_{H_1,1}u^{*}(q)-3\epsilon$.
Since this inequality holds for any $\epsilon >0$ , we obtain
$$u^{*}(q) \geq T^-_{H_1,1}u^{*}(q).$$

Thus, Lemma 3.2 follows.

\end{proof}
 Now the proof of Proposition 3.2 is completed.
\end{proof}

Furthermore, we also have
\begin{prop}
$\mathcal{S}^{-}_{H_1}\cap \mathcal{S}^{-}_{H_2} \neq \emptyset, \mathcal{S}^{+}_{H_1}\cap \mathcal{S}^{+}_{H_2} \neq \emptyset.$
\end{prop}
\begin{proof}
We also only prove the first relation.

Choose $u \in \mathcal{S}^{-}_{H_1,0}\cap \mathcal{S}^{-}_{H_2,0}$,  we will prove that $T^-_{H_1,0,t}u=T^-_{H_2,0,t}u$ for any $t \in (0,1]$.

For $t=1$, it follows directly from the definitions of $\mathcal{S}^{-}_{H_1,0}$ and  of $\mathcal{S}^{-}_{H_2,0}$.

Now we consider the case $t \in (0,1)$. Otherwise, there exists a point $q_0 \in M$ and a rational number $t_0 \in (0,1)$ such that
$$T^-_{H_1,0,t_0}u(q_0)\neq T^-_{H_2,0,t_0}u(q_0).$$

Without loss of generality, we may assume that $$T^-_{H_1,0,t_0}u(q_0) < T^-_{H_2,0,t_0}u(q_0).$$

Assume that $k_0$ is the smallest positive integer such that $k_0t_0 \in \mathbb{Z}^+$. Then,

\begin{eqnarray*}&&T^-_{H_1,0,2t_0}u(q_0)
\\&=&T^-_{H_1,t_0,2t_0}T^-_{H_1,0,t_0}u(q_0)\\
&<&T^-_{H_1,t_0,2t_0}T^-_{H_2,0,t_0}u(q_0)\\
&=&T^-_{H_2,t_0,2t_0}T^-_{H_1,0,t_0}u(q_0)\\
&<&T^-_{H_2,t_0,2t_0}T^-_{H_2,0,t_0}u(q_0)\\
&=&T^-_{H_2,0,2t_0}u(q_0),
\end{eqnarray*}
here, the first equality and the third equality follow from Markov property of Lax-Oleinik operators; the second equality follows from commuting property of Lax-Oleinik operators (Theorem 1); the two inequalities follow from the order-preserving property of Lax-Oleinik operators.

By induction to the $(k_0-1)$th step of this procedure, we obtain

$$
u(q_0)=T^-_{H_1,0,k_0t_0}u(q_0)
<T^-_{H_2,0,k_0t_0}u(q_0)=u(q_0).
$$
This contradiction proves Proposition 3.3.
\end{proof}

\begin{prop}
$\tilde{H_2}|_{\stackrel{\scriptscriptstyle
\ast}{A}^e_{H_1}}=\alpha_{H_2}(0); \tilde{H_1}|_{\stackrel{\scriptscriptstyle
\ast}{A}^e_{H_2}}=\alpha_{H_1}(0)$.
\end{prop}
\begin{proof}
Throughout this paper, $\overline{d\phi}$ denotes the closure of the set of $$\{\big{(}(q,[t]),d_{(q,[t])}\phi \big{)}|\phi
\text{ is differentiable at } (q,[t])\}.$$

Choose $\phi^* \in \mathcal{S}^{-}_{H_1} \cap \mathcal{S}^{-}_{H_2}$, we have
$\alpha_{H_2}(0)=\tilde{H_2}|_{\overline{d\phi^*}}$. Since $\stackrel{\scriptscriptstyle
\ast}{A}^e_{H_1} \subseteq \overline{d\phi^*}$, the first equality holds.
The second equality follows analogously.
\end{proof}

\section{Proof of Theorem 2}
Let $\phi \in \mathcal{S}^-_{H_1}$. Since $\phi$ is Lipschitz, $\phi$ is differentiable almost everywhere (with respect to Lebesgue measure (i.e., volume induced by the product Reimennian metric on $M \times \mathbb{T}$, the first factor is the Riemannian metric on $M$, the second one is the Euclidean metric on $\mathbb{T}$. In fact, which Riemannian metric we choose does not matter at all, since any the two volumes as measures with respect to any two Riemannian metrics on a closed Riemannian manifold are absolutely continuous with respect to each other), by Rademacher's theorem.  Now choose any differentiable point $(q_0,[t_0])$ of $\phi$, then there exists a unique backward minimizer $\gamma:(-\infty, t_0] \rightarrow M$ with $\gamma(t_0)=q_0$ and $\mathcal{L}_{H_1}(\dot{\gamma}(t)) \in d_q\phi$. Then the associated trajectory of $\tilde{H_1}$ is $$(\phi^s_{H_1(q_0,d_q\phi(q_0,t_0),t_0)},\kappa_0+H(q_0,d_q\phi(q_0,t_0),[t_0])-H(\phi^s_H(q,d_q\phi(q_0,t_0),[t_0]))),$$ here $s \in (-\infty,0]$, $\kappa_0+H_1(q_0,d_q(q_0,t_0),t_0)=\alpha_{H_1}(0)$. Since $\{\tilde{H_1}, \tilde{H_2}\}=0$, $\tilde{H_2}$ is constant along this trajectory, hence on its closure. Note that the limit set of  $\phi^s_{H_1(q_0,d_q(q_0,t_0),[t_0])}$ lies in the Aubry set $\stackrel{\scriptscriptstyle
\ast}{A}_{H_1}$, thus the limit set of $$(\phi^s_{H_1(q_0,d_q(q_0,t_0),t_0)},\kappa_0+H(q_0,d_q\phi(q_0,t_0),[t_0])-H(\phi^s_H(q,d_q\phi(q_0,t_0),[t_0])))$$ lies in the extended Aubry set $\stackrel{\scriptscriptstyle
\ast}{A}^e_{H_1}$, by the choice of $\kappa_0$. By Proposition 3.4, we have  $\tilde {H_2}|_{\overline{d\phi}}=\alpha_{H_2}(0)$. Thus,
\begin{equation}
H_2(q,d_q\phi(q,[t]),[t])+d_t\phi(q,[t])=\alpha_{H_2}(0)
\end{equation}
at any differentiable point  $(q,[t])$ of $\phi$.
By the relation of time 1-periodic viscosity solutions and weak KAM solutions, we know that $\phi$ is  a viscosity solution of Hamilton-Jacobi equation
$$H_1(q,d_q\phi,[t])+d_t\phi=\alpha_{H_1}(0).$$
Hence, $\phi$ is a \emph{semi-concave} (not necessarily with linear modulus, see \cite{CS} for the details) function on $M \times \mathbb{T}$. So, by (1), clearly we have $\phi$ is also a viscosity supersolution of
$$H_2(q,d_q\phi(q,[t]),[t])+d_t\phi(q,[t])=\alpha_{H_2}(0).$$

The deduction  that $\phi$ is viscosity subsolution of

$$H_2(q,d_q\phi(q,[t]),[t])+d_t\phi(q,[t])=\alpha_{H_2}(0)$$

from (1) also appeared in \cite{CS}, by the fact that the set of \emph{upper-differential} (this terminology was called subdifferential in \cite{CS}) is the convex hull of the set of \emph{reachable gradient}, and the fact that $\tilde {H_2}$ is convex (not necessary to be strictly convex) with respect to the variable $(p,\kappa)$ in our case. The reader is advised  to look at \cite{CS} for more details.

Now we have proved that $\phi$ is a backward weak KAM solution of
$$H_1(q,d_q\phi,[t])+d_t\phi=\alpha_{H_1}(0)$$
if and only if $\phi$ is a backward weak KAM solution of
$$H_2(q,d_q\phi,[t])+d_t\phi=\alpha_{H_2}(0),$$
That is $\mathcal{S}^-_{H_1}=\mathcal{S}^-_{H_2}$.

Analogously,  by the symmetric Hamiltonian method (as in \cite{Fa3},\cite{CL}, with slight modifications) we have that $\phi$ is a froward weak KAM solution of
$$H_1(q,d_q\phi,[t])+d_t\phi=\alpha_{H_1}(0)$$
if and only if $\phi$ is a forward weak KAM solution of
$$H_2(q,d_q\phi,[t])+d_t\phi=\alpha_{H_2}(0).$$
More precisely, let $\check H(q,p,[t])=H(q,-p,[-t])$, it is easy to check that
$\phi \in \mathcal{S}^+_{H}$ if and only if $\phi \in \mathcal{S}^-_{\check H}$, and $[H_1,H_2]=0$ if and only if $[\check {H_1}, \check {H_2}]=0$. Thus the discussion above  in this section applies.

Up to now, Theorem 2 is proved.

\section{Proof of Theorem 3}

Let us begin this section with a definition.
\begin{defi}
For a Tonelli Hamiltonian $H$, we
say that $\phi_-\in \mathcal{S}^-_{H}$ and $\phi_+\in
\mathcal{S}^+_{H}$ are conjugate with respect to $H$ if
$\phi_-=\phi_+$ on the projected Mather set $M_{H}$. If $\phi_-$ and
$\phi_+$ are conjugate with respect to $H$, we also denote this relation
 by $\phi_- \sim_H \phi_+$.
\end{defi}
Based on this definition, we can express equivalent definitions
\cite{Fa3} of $A_{H,c}$ and of $N_{H,c}$ as follows:
$$A_{H}=\cap \Big{\{}(q,[t]):\phi_-(q,[t])=\phi_+(q,[t]), \text{ where } \phi_-\in
\mathcal{S}^-_{H}, \phi_+\in \mathcal{S}^+_{H}, \phi_- \sim_H
\phi_+\Big{\}}$$ and
$$N_{H}=\cup \Big{\{}(q,[t]):\phi_-(q,[t])=\phi_+(q,[t]), \text{ where } \phi_-\in
\mathcal{S}^-_{H}, \phi_+\in \mathcal{S}^+_{H}, \phi_- \sim_H
\phi_+\Big{\}}.$$
Consequently, we also have \cite{Fa3}
$$\stackrel{\scriptscriptstyle
\ast}{A}^e_{H}=\cap\Big{\{}(q,p,[t],\kappa)|\phi_-(q,[t])=\phi_+(q,[t]),(p,\kappa)=d\phi_-=d\phi_+,\phi_-\in
\mathcal{S}^-_{H}, \phi_+\in \mathcal{S}^+_{H}, \phi_- \sim_H
\phi_+\Big{\}},$$
$$\stackrel{\scriptscriptstyle
\ast}{N}^e_{H}=\cup \Big{\{}(q,p,[t],\kappa)|\phi_-(q,[t])=\phi_+(q,[t]),(p,\kappa)=d\phi_-=d\phi_+,\phi_-\in
\mathcal{S}^-_{H}, \phi_+\in \mathcal{S}^+_{H}, \phi_- \sim_H
\phi_+\Big{\}},$$
$$\stackrel{\scriptscriptstyle
\ast}{A}_{H}=\cap\Big{\{}(q,p,[t])|\phi_-(q,[t])=\phi_+(q,[t]),p=d_q\phi_-=d_q\phi_+,\phi_-\in
\mathcal{S}^-_{H}, \phi_+\in \mathcal{S}^+_{H}, \phi_- \sim_H
\phi_+\Big{\}},$$  and
$$\stackrel{\scriptscriptstyle
\ast}{N}^e_{H}=\cup \Big{\{}(q,p,[t])|\phi_-(q,[t])=\phi_+(q,[t]),p=d_q\phi_-=d_q\phi_+,\phi_-\in
\mathcal{S}^-_{H}, \phi_+\in \mathcal{S}^+_{H}, \phi_- \sim_H
\phi_+\Big{\}}.$$
\begin{prop}
If $[H_1, H_2]=0$, then $\stackrel{\scriptscriptstyle
\ast}M^e_{H_1} \subseteq \stackrel{\scriptscriptstyle
\ast}A^e_{H_2}, \stackrel{\scriptscriptstyle \ast} M^e_{H_2}
\subseteq \stackrel{\scriptscriptstyle \ast}A^e_{H_1}$.
\end{prop}
\begin{proof}
We only need to show that $\stackrel{\scriptscriptstyle
\ast}M^e_{H_1} \subseteq \stackrel{\scriptscriptstyle \ast}A^e_{H_2}
$, by the symmetry of $H_1$ and $H_2$. By weak KAM theory
\cite{Fa3}, we have
$$\stackrel{\scriptscriptstyle
\ast}M^e_{H_1} \subseteq \cap_{\phi_- \in \mathcal{S}^-_{H_1}}
\big{\{}d\phi_-\big{\}}=\cap_{\phi_- \in\mathcal{S}^-_{H_2}}
\big{\{}d\phi_-\big{\}}.$$

Now we need the following lemma:
\begin{lem}
$\stackrel{\scriptscriptstyle \ast}M^e_{H_1}$ is invariant under the flow  $\phi^t_{\tilde {H_2}}$.
\end{lem}
\begin{proof}
The proof follows from the symplectic invariance of Mather set \cite{Be1}, as cued by Sorrentino \cite{So}. In fact, for any fixed $t \in \mathbb{R}$, $\phi^t_{\tilde {H_2}}$ is a Hamiltonian diffeomorphism on $T^*(M \times \mathbb{T})$. Moreover, $\tilde{H_1}(\phi^t_{\tilde {H_2}})=\tilde{H_1}$, since $ \tilde {H_1}$ is constant on the trajectory of $\tilde {H_2}$. In other words, ${\phi^t_{\tilde {H_2}}}^*\tilde {H_1}=\tilde {H_1}$. By the result of Bernard \cite{Be1}, we have $$\phi^t_{\tilde {H_2}}(\stackrel{\scriptscriptstyle \ast}M^e_{{\phi^t_{\tilde {H_2}}}^*H_1})=\stackrel{\scriptscriptstyle \ast}M^e_{H_1},$$
here ${\phi^t_{\tilde {H_2}}}^*{H_1}$ is the function defined as:
$${\phi^t_{\tilde {H_2}}}^*{H_1}(q_0,p_0,[t_0])=\tilde H_1(\phi^{-t}_{\tilde {H_2}}(q_0,p_0,[t_0],\kappa_0))-\kappa_0.$$
Note that the definition of ${\phi^t_{\tilde {H_2}}}^*{H_1}$ is independent of the choice of $\kappa_0$.

Since ${\phi^t_{\tilde {H_2}}}^*\tilde {H_1}=\tilde {H_1}$, we have ${\phi^t_{\tilde {H_2}}}^* H_1=\ H_1$ and
$$\phi^t_{\tilde {H_2}}(\stackrel{\scriptscriptstyle \ast}M^e_{H_1})=\stackrel{\scriptscriptstyle \ast}M^e_{H_1}$$
consequently.
\end{proof}

In \cite{Mas},  the following lemma appeared:
\begin{lem}
For any two points $(q_0,[t_0]),(q_1,[t_1]) \in M \times \mathbb{T}$, we have the following equality

$$h_{H}((q_0,[t_0]),(q_1,[t_1]))=\sup\Big{\{}\phi_-(q_1,[t_1])-\phi_+(q_0,[t_0]):\phi_-\in
\mathcal{S}^-_{H}, \phi_+\in \mathcal{S}^+_{H}, \phi_- \sim_H
\phi_+\Big{\}}.$$ Moreover, for any given two points $(q_0,[t_0]),(q_1,[t_1]) \in M\times \mathbb{T}$, this
supremum is actually attained.
\end{lem}

As a consequence of this lemma, together with the definition of
conjugate pair of weak KAM solutions,  we have
\begin{cor}
$$h_{H}((q_0,[t_0]),(q_1,[t_1]))=\sup \Big{\{}\phi_-(q_1,[t_1])-\phi_-(q_0,[t_0]):\phi_-\in
\mathcal{S}^-_{H}\Big{\}}$$ for any two points $(q_0,[t_0]),(q_1,[t_1]) \in A_{H}$. Moreover, for any given two points $(q_0,[t_0]),(q_1,[t_1]) \in A_{H}$, this
supremum is actually attained.
\end{cor}

Choose any point $(q_0,p_0,[t_0],\kappa_0) \in \stackrel{\scriptscriptstyle
\ast}M^e_{H_1}$, we will show that $(q_0,p_0,[t_0],\kappa_0) \in
\stackrel{\scriptscriptstyle \ast}A^e_{H_2}$.

First, we will show that $\pi \circ \phi^t_{\tilde {H_2}}(q_0,p_0,[t_0],\kappa_0)$ is a
minimizer with respect to $L_{H_2}$. For any $t_1 < t_2 \in
\mathbb{R}$, we have
\begin{eqnarray*}
&&\int^{t_2}_{t_1}(L_{H_2}+\alpha_{H_2}(0))(\pi \circ
\phi^t_{H_2}(q_0,p_0,[t_0]), \frac{d}{dt}({\pi \circ \phi}^t_{H_2}(q_0,p_0,[t_0])),[t+t_0])dt\\
&=& \int^{t_2}_{t_1}L_{H_2}(\pi \circ
\phi^t_{H_2}(q_0,p_0,[t_0]), \frac{d}{dt}({\pi \circ \phi}^t_{H_2}(q_0,p_0, [t_0])),[t+t_0])dt\\
&+&\int^{t_2}_{t_1}\tilde{H_2}(\phi^t_{\tilde {H_2}}(q_0,p_0,[t_0],\kappa_0))dt\\
&=& \int^{t_2}_{t_1}L_{H_2}(\pi \circ
\phi^t_{H_2}(q_0,p_0,[t_0]), \frac{d}{dt}({\pi \circ \phi}^t_{H_2}(q_0,p_0,[t_0])),[t+t_0])dt\\
&+&\int^{t_2}_{t_1}(H_2(\phi^t_{H_2}(q_0,p_0,[t_0]))+\kappa(t))dt\\
&=&\int^{t_2}_{t_1}(pdq+\kappa dt) (\phi^t_{\tilde {H_2}}(q_0,p_0,[t_0],\kappa_0))dt\\
 &=&\phi_-(\pi \circ \phi^{t_2}_{\tilde {H_2}}(q_0,p_0,[t_0],\kappa_0))-\phi_-(\pi \circ
\phi^{t_1}_{\tilde{H_2}}(q_0,p_0,[t_0],\kappa_0))\\
&\leq & \phi_{H_2}((\pi \circ \phi^{t_1}_{\tilde {H_2}}(q_0,p_0,[t_0],\kappa_0),[t_0+t_1]), (\pi \circ \phi^{t_2}_{\tilde {H_2}}(q_0,p_0,[t_0],\kappa_0),[t_0+t_2])),
\end{eqnarray*}
where $\phi_- \in
\mathcal{S}^-_{H_1}(=\mathcal{S}^-_{H_2})$; $\kappa(t)$ is the last component of the flow $\phi^t_{\tilde{H_2}}(q_0,p_0,[t_0],\kappa_0)$; the first
equality follows from the fact that
$\tilde {H_2}(\phi^t_{\tilde {H_2}}(q_0,p_0,[t_0],\kappa_0))=\alpha_{H_2}(0)$, since $\tilde {H_2}|_{\stackrel{\scriptscriptstyle \ast}A^e_{H_1}}=\alpha_{H_2}(0)$; the fourth
equality follows from the fact that
 $$\phi^t_{\tilde {H_2}}(q_0,p_0,[t_0],\kappa_0) \in \stackrel{\scriptscriptstyle \ast}M^e_{H_1} \subset \cap_{\phi_- \in\mathcal{S}^-_{H_1}}\big{\{}d\phi_-\big{\}}=\cap_{\phi_- \in\mathcal{S}^-_{H_2}}
\big{\{}d\phi_-\big{\}};$$
the inequality follows from the fact
that $\phi_- \in \mathcal{S}^-_{H_2}$ (Since $\phi_- \in \mathcal{S}^-_{H_2}$, then $\phi_- \prec L_{H_2}+\alpha_{H_2}(0)$). Thus,
$$\stackrel{\scriptscriptstyle \ast}M^e_{H_1} \subseteq
\stackrel{\scriptscriptstyle \ast}N^e_{H_2}.$$

Hence, we only need
to show that $\rho_{H_2}((q_{\alpha},[t_{\alpha}]),(q_{\omega},[t_{\omega}]))=0$, for any point
$(q_{\alpha},[t_{\alpha}])$ lies in the $\alpha$-limit set and any point $(q_{\omega},[t_{\omega}])$ lies in
the $\omega$-limit set of $\pi \circ \phi^t_{H_2}(q_0,p_0)$ in $M\times \mathbb{T}$.
Clearly, both  $(q_{\alpha},[t_{\alpha}])$ and $(q_{\omega},[t_{\omega}])$ lie in $A_{H_2}$,
since $\pi \circ \phi^t_{H_2}(q_0,p_0,[t_0])$ is a minimizer with respect
to $L_{H_2}$. Thus, we can use the formula in Corollary 5.1 to
calculate $\rho_{H_2}((q_{\alpha},[t_{\alpha}]),(q_{\omega},[t_{\omega}]))$:
\begin{eqnarray*}
&&\rho_{H_2}((q_{\alpha},[t_{\alpha}]),(q_{\omega},[t_{\omega}]))\\
&=&h_{H_2}((q_{\alpha},[t_{\alpha}]),(q_{\omega},[t_{\omega}]))+
h_{H_2}((q_{\omega},[t_{\omega}]),(q_{\alpha},[t_\alpha]))\\
&=&\sup \Big{\{}\phi_-(q_{\omega},[t_{\omega}])-\phi_-(q_{\alpha},[t_{\alpha}]):\phi_-\in
\mathcal{S}^-_{H_2}\Big{\}}\\&+&\sup
\Big{\{}\psi_-(q_{\alpha},[t_{\alpha}])-\psi_-(q_{\omega},[t_{\omega}]):\psi_-\in
\mathcal{S}^-_{H_2}\Big{\}}\\&=& \sup
\Big{\{}\phi_-(q_{\omega},[t_{\omega}])-\phi_-(q_{\alpha},[t_{\alpha}]):\phi_-\in
\mathcal{S}^-_{H_1}\Big{\}}\\&+&\sup
\Big{\{}\psi_-(q_{\alpha},[t_{\alpha}])-\psi_-(q_{\omega},[t_{\omega}]):\psi_-\in
\mathcal{S}^-_{H_1}\Big{\}}\\&=&
h_{H_1}((q_{\alpha},[t_{\alpha}]),(q_{\omega},[t_{\omega}]))+
h_{H_1}((q_{\omega},[t_{\omega}]),(q_{\alpha},[t_{\alpha}]))\\
&=&\rho_{H_1}((q_{\alpha},[t_{\alpha}]),(q_{\omega},[t_{\omega}])),
\end{eqnarray*}
where, the second and the fourth qualities  follow from Corollary 5.1; the third
equality follows from the fact that
$\mathcal{S}^-_{H_1}=\mathcal{S}^-_{H_2}$.

Now we claim that $\rho_{H_1}((q_{\alpha},[t_{\alpha}]),(q_{\alpha},[t_{\omega}]))=0$.  Since
$(q_{\alpha},[t_{\alpha}])$ and $(q_{\omega},[t_{\omega}])$ lie in the $\alpha$-limit set and
$\omega$-limit set of $\pi \circ \phi^t_{H_2}(q_0,p_0,[t_0])$
respectively, there exist $t_i, t_k \rightarrow +\infty$ as $i, k
\rightarrow +\infty$, such that $$ \pi \circ
\phi^{-t_i}_{H_2}(q_0,p_0,[t_0]) \rightarrow (q_{\alpha},[t_{\alpha}]), \,\,\,\,\, \pi
\circ \phi^{t_k}_{H_2}(q_0,p_0,[t_0]) \rightarrow (q_{\omega},[t_{\omega}]).$$  Now we
have $$\rho_{H_1}(\pi \circ \phi^{-t_i}_{H_2}(q_0,p_0),\pi \circ
\phi^{t_k}_{H_2}(q_0,p_0,[t_0]))=0,$$ since  $\pi \circ
\phi^{-t_i}_{H_2}(q_0,p_0,[t_0])$ and $\pi \circ
\phi^{t_k}_{H_2}(q_0,p_0,[t_0])$ can be connected by a $C^2$ curve $\pi
\circ \phi^t_{H_2}(q_0,p_0,[t_0])$ which lies in $M_{H_1}$ and
$\rho_{H_1}$ satisfies
$$\rho_{H_1}((q_0,[t_0]),(q_1,[t_1]) \leq C(d(q_0,q_1)^2+|[t_1]-[t_0]|^2)$$ for each $(q_0,[t_0]),(q_1,[t_1]) \in
A_{H_1}$ \cite{Mat2}, where $C$ is a constant, $d$ is the
distance induced by the Riemannian metric, and $|[t_1]-[t_0]|$ means the distance of $[t_0]$ and $[t_1]$ on the circle $\mathbb{R}/\mathbb{Z}$. So,
$$\rho_{H_1}((q_{\alpha},[t_{\alpha}]),(q_{\omega},[t_{\omega}]))= 0,$$
by taking a limit.

Thus, Proposition 5.1 follows.
\end{proof}
\begin{prop}
Assume that $[H_1, H_2]=0$. Let $\phi_- \in
\mathcal{S}^-_{H_1}=\mathcal{S}^-_{H_2},$ $\phi_+ \in
\mathcal{S}^+_{H_1}=\mathcal{S}^+_{H_2}.$ Then $\phi_-$ and
$\phi_+$ are conjugate with respect to $H_1$ if and only if $\phi_-$ and
$\phi_+$ are conjugate with respect to $H_2$, i.e., $\phi_- \sim_{H_1} \phi_+
\iff \phi_- \sim_{H_2}\phi_+.$
\end{prop}

\begin{proof}
It is a direct consequence of Proposition 5.1  and the fact
\cite{Fa3} that
$$A_{H}=\cap \big{\{}(q,[t])| \phi_-(q,[t])=\phi_+(q,[t]), \phi_-\in
\mathcal{S}^-_{H}, \phi_+\in \mathcal{S}^+_{H}, \phi_- \sim_H
\phi_+\big{\}}.$$
\end{proof}

In \cite{Fa1},\cite{Fa3}, Fathi showed that
\begin{lem}
$$B_{H}(q,[t])=\sup \Big{\{}\phi_{-}(q,[t])-\phi_{+}(q,[t])|\phi_-\in
\mathcal{S}^-_{H}, \phi_+\in \mathcal{S}^+_{H}, \phi_- \sim_H
\phi_+\Big{\}},$$ and, moreover, the supremum is attained  for each
$(q,[t])\in M \times \mathbb{T}$.
\end{lem}
 Recall that
\begin{eqnarray*}
b_{H}(q,[t])&=&\inf_{\xi,\zeta \in A_{H}}\Big{\{}h_{H}(\xi,
(q,[t]))+h_{H}((q,[t]),\zeta) -h_{H}(\xi,\zeta)\Big{\}}\\
&\Big{(}&=\min_{\xi,\zeta \in A_{H}}\Big{\{}h_{H}(\xi,
(q,[t]))+h_{H}((q,[t]),\zeta) -h_{H}(\xi,\zeta)\Big{\}}\Big{)}.
\end{eqnarray*}

 In fact, the autonomous version of following lemma also appeared in \cite{Fa1}, for the proof, see also \cite{CL}. We can obtain the time-periodic version with slight modifications:

\begin{lem}
$$b_{H}(q,[t])=\inf\Big{\{}u_{-}(q,[t])-u_{+}(q,[t])|u_-\in
\mathcal{S}^-_{H}, u_+\in \mathcal{S}^+_{H}, u_- \sim_H
u_+\Big{\}}.$$ Moreover,  the infimum is attained for each $(q,[t])\in M \times \mathbb{T}$.
\end{lem}

Now Theorem 3 follows from Theorem 1, Proposition 5.2, Lemma 5.3
and Lemma 5.4.

Clearly, we also have
\begin{cor}
$(\bar{A}_{H_1},\rho_{H_1})$ and $(\bar{A}_{H_2},\rho_{H_2})$ are isometric.
\end{cor}

\section{Proof of Theorem 4}
Since
$$\stackrel{\scriptscriptstyle
\ast}{A}^e_{H}=\cap\Big{\{}(q,p,[t],\kappa)|\phi_-(q,[t])=\phi_+(q,[t]),(p,\kappa)=d\phi_-=d\phi_+,\phi_-\in
\mathcal{S}^-_{H}, \phi_+\in \mathcal{S}^+_{H}, \phi_- \sim_H
\phi_+\Big{\}}$$ and
$$\stackrel{\scriptscriptstyle
\ast}{N}^e_{H}=\cup \Big{\{}(q,p,[t],\kappa)|\phi_-(q,[t])=\phi_+(q,[t]),(p,\kappa)=d\phi_-=d\phi_+, \phi_-\in
\mathcal{S}^-_{H}, \phi_+\in \mathcal{S}^+_{H}, \phi_- \sim_H
\phi_+\Big{\}}.$$

 Theorem 4 follows from Theorem 1 and Proposition 3.2.

 \begin{cor}
 Let $H_1$ and $H_2$ be two Tonelli Hamiltonians. If $[H_1,H_2]=0$, then $$\stackrel{\scriptscriptstyle
\ast}{A}_{H_1}=\stackrel{\scriptscriptstyle
\ast}{A}_{H_2},\,\,\, A_{H_1}=A_{H_2}$$and
$$\stackrel{\scriptscriptstyle
\ast}{N}_{H_1}=\stackrel{\scriptscriptstyle
\ast}{N}_{H_2}, \,\,\, N_{H_1}=N_{H_2}.$$
 \end{cor}

\section{A conjecture}
For an autonomous Tonelli Hamiltonian $H$, Bernard \cite{Be2} proved the existence of $C^{1,1}$ critical subsolution, i.e.,
$$H(q,d_qu)=\alpha_{H}(0)$$
admits at least one  $C^{1,1}$ subsolution.  Based on this result, still in the autonomous case, it was proved that two associated Hamilton-Jacobi equations have at least one  common $C^{1,1}$ (critical) subsolution, if the two  Tonelli Hamiltonians we considered  are commuting \cite{CL},\cite{Za}. For a time-periodic Tonelli Hamiltonian,  Massart \cite{Mas} has proved the existence of a $C^1$ critical subsolution. Also, in a work in preparation \cite{Be3}, a student of Bernard proved the existence of $C^{1,1}$ critical subsolution. Namely, Hamilton-Jacobi equation
$$H(q,d_q\phi,[t])+d_t\phi=\alpha_H(0)$$
has at least one $C^{1,1}$ subsolution. In view of these previous results, it is natural to pose the following conjecture, as a counterpart of Theorem 6 in \cite{CL}:

\begin{conj}
Let $H_1, H_2$ be two Tonelli Hamiltonians. If $[H_1,H_2]=0$, then there exists at least one  common  $C^{1,1}$ subsolution for
$$H_1(q,d_q\phi, [t])+d_t\phi=\alpha_{H_1}(0)$$
and
$$H_2(q,d_q\phi, [t])+d_t\phi=\alpha_{H_2}(0).$$
\end{conj}

\section{A final remark}
For the time-periodic case, in general we cannot obtain $[H_1+H_2,H_1]=0$ or $[H_1+H_2,H_2]=0$ under the condition that $[H_1,H_2]=0$,  due to the influence of time factor. So, we cannot obtain the  dynamical information of $H_1+H_2$ directly, as in the autonomous case \cite{CL}. It seems to the author that when $[H_1,H_2]=0$,  the relations between $\stackrel{\scriptscriptstyle
\ast}{A}_{H_1+H_2} \text{\,\, and\,\,}\stackrel{\scriptscriptstyle
\ast}{A}_{H_1}(=\stackrel{\scriptscriptstyle
\ast}{A}_{H_2}),$
$\stackrel{\scriptscriptstyle
\ast}{N}_{H_1+H_2} \text{\,\, and\,\,}\stackrel{\scriptscriptstyle
\ast}{N}_{H_1}(=\stackrel{\scriptscriptstyle
\ast}{N}_{H_2}),$
$B_{H_1+H_2}$ and $B_{H_1}(=B_{H_2})$, $b_{H_1+H_2}$ and $b_{H_1}(=b_{H_2})$ are  more complicated than the ones in the autonomous case.

\begin{ack}
 This paper was completed when the  author
visited Albert-Ludwigs-Universit\"{a}t  Freiburg as a postdoctoral researcher, supported
by a fellowship from the Alexander Von Humboldt Foundation. The
 author would like to thank Professor V. Bangert and Mathematisches Institut der Albert-Ludwigs-Universit\"{a}t Freiburg for hospitality. The author also would like to thank Dr. L. Zhao for some helpful discussions.
\end{ack}

\end{document}